\documentclass[12pt]{article}
\usepackage{amsmath,amssymb,amsthm,graphicx,hyperref,mathptmx,color}
\usepackage{geometry}

\newtheorem{theorem}{Theorem}
\newtheorem{proposition}[theorem]{Proposition}
\theoremstyle{remark}
\newtheorem{remark}{Remark}
\newcommand{\bl}[1]{{#1}}
\let\le=\leqslant
\let\ge=\geqslant
\DeclareMathOperator{\argmax}{\mathrm{argmax}}

\begin{document}
\title{Layerwise computability and image randomness}
\author{Laurent Bienvenu\thanks{Laboratoire d'Informatique, de Robotique et de Micro\'electronique de Montpellier, CNRS, Universit\'e de Montpellier, \url{laurent.bienvenu@computability.fr}. Supported by ANR RaCAF grant.} \and  Mathieu Hoyrup\thanks{Laboratoire Lorraine de Recherche en Informatique et ses applications, \url{mathieu.hoyrup@inria.fr}} \and Alexander Shen \thanks{Visiting researcher, National Research University Higher School of Economics (HSE), Faculty of Computer Science, Kochnovskiy Proezd 3, Moscow, 125319, Russia;  Laboratoire d'Informatique, de Robotique et de Micro\'electronique de Montpellier, CNRS, Universit\'e de Montpellier, \url{alexander.shen@lirmm.fr}. Supported by ANR RaCAF grant.}}

\date{}
\maketitle

\begin{abstract}
Algorithmic randomness theory starts with a notion of an individual random object. To be reasonable, this notion should have some natural properties; in particular, an object should be random with respect to image distribution if and only if it has a random preimage. This result (for computable distributions and mappings, and Martin-L\"of randomness) was known for a long time (folklore); in this paper we prove its natural generalization for layerwise computable mappings, and discuss the related quantitative results.
\end{abstract}

\section{Introduction}
% Here we explain the philosophical origins of the question, give some definitions of transformations, and survey the results

Consider some random process, like coin tossing, that generates an infinite sequence of bits (zeros and ones). Assume that we have some theoretical model that describes this process. Such a model is some probability distribution $P$ on the space of possible outcomes; for example, fair coin tossing is modelled by \bl{the} uniform distribution on the space of infinite bit sequences, also known as Cantor space. Now suppose that we get, in addition to this theoretical model, some experimental data, a sequence $\omega$. The natural question arises: is this data consistent with the model? Is it plausible that a random process with output distribution $P$ produced $\omega$, or the conjecture about probability distribution should be rejected? Sometimes it should --- for example, if more than $90\%$ bits in any prefix of $\omega$ are zeros, the fair coin assumption does not look acceptable. The same is true if the sequence $\omega$ turns out to be a binary representation of $\pi$. On the other hand, some $\omega$ should be considered as plausible outputs (otherwise the conjecture is always rejected). But where is the boundary line?

Classical probability theory does not provide any answer to this question. This is what \emph{algorithmic randomness theory} is about. This theory defines a notion of a random sequence with respect to a probability distribution. More precisely, for a probability distribution $P$ on the Cantor space,  it defines a subset of the Cantor space whose elements are called \emph{random sequences with respect to $P$}, or \emph{$P$-random sequences}. In fact, there are several definitions; the most basic and popular one is called \emph{Martin-L\"of randomness} and was introduced in~\cite{martinlof}, but there are some others. Having such a notion, we can say that we reject the model (distribution) $P$ if the experimental data, the sequence $\omega$, is not random with respect to~$P$. This interpretation, however, comes with some obligations: the notion of randomness should have some properties to make it reasonable.  In this note we consider one of these properties: \emph{image randomness}. 

Imagine that we have a machine that consists of two parts: (1)~a random process that generates some bit sequence $\alpha$, and has distribution $P$, and (2)~some algorithmic process $M$ that gets $\alpha$ as input and produces some output sequence $\beta$. We may assume that the machine $M$ has an input read-only tape where $\alpha$ is written, unlimited memory (say, infinite work tapes) and write-only one-directional output tape where the bits of $\beta$ are written sequentially. Such a machine determines a mapping of the Cantor space into the space of finite and infinite binary sequences. When $M$ is applied to a random input having distribution $P$, the output sequence is a random variable. The distribution of this random variable is what is usually called a \emph{continuous semimeasure}, but we are mainly interested in the case when the output sequence is infinite with $P$-probability one. In this case the output distribution is some distribution $Q$ on Cantor space (the image of $P$). Such a scheme is used, for example, when we want to emulate some distribution $Q$ using a fair coin that produces independent uniformly distributed random bits. In this case $P$ is the uniform distribution on the Cantor space.
    
\begin{figure}[h]
\begin{center}
\includegraphics[scale=1]{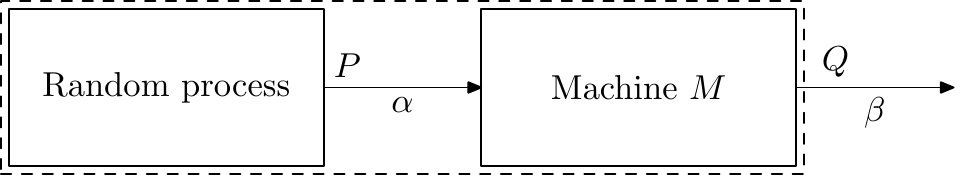}
\end{center}
\end{figure}
   
Now imagine that we observe some sequence $\beta$ as the output of such a composite random process. When does it look plausible? There are two possible answers. The first one: ``when \emph{$\beta$ is random with respect to $Q$}'', since $Q$ is the output distribution of the composite process. The second one: ``when there exists a plausible reconstruction of what happened inside the black box'', i.e., when \emph{there exists some $\alpha$ that is random with respect to $P$ and is mapped to $\beta$ by $M$}. To make the philosophical interpretation consistent, these two properties should be equivalent: for a satisfactory notion of algorithmic randomness the image of the set of algorithmically random with respect to $P$ sequences should coincide with the set of algorithmically random with respect to $Q$ sequences, where $Q$ is the image of $P$ under the transformation.

This property can be split into two parts. One part says that $M$-image of every $P$-random sequence is infinite and $Q$-random. This part can be called \emph{randomness conservation}. The other part claims that every $Q$-random sequence is an $M$-image of some $P$-random sequence: \emph{randomness cannot be created out of nothing}.

It turns out that this property holds for Martin-L\"of's definition of randomness and computable $P$. Note that Martin-L\"of's definition is applicable only to computable distributions; if $P$ is computable and $M(\alpha)$ is infinite with $P$-probability $1$, then $Q$ is also computable, so Martin-L\"of's definition can be applied to both $P$ and $Q$. This is a folklore result known for a long time; the third author remembers discussing it with Andrej Muchnik around 1987  (it seems that there was even a short note written by A.S. about that in the proceedings of the First Bernoulli congress in Uzbekistan, then USSR; most probably nobody read it, and we cannot locate this volume now). It was independently discovered by many people, sometimes for some special cases. 

Here is one example where this property is useful. Let us prove that for every random sequence $\omega$ with respect to the uniform Bernoulli measure $B_{1/2}$ (independent trials, ones have probability $1/2$) there exist a sequence $\tau$ that is random with respect to $B_{1/3}$ (Bernoulli distribution where ones have probability $1/3$) that can be obtained from $\omega$ by changing some ones to zeros. Indeed, let us imagine a random process  that first generates a sequence of random\footnote{A sequence of reals can be encoded as a two-dimensional table of bits, if each real is replaced by its binary representation; then this two-dimensional table can be rearranged into a sequence of bits, and the corresponding distribution on the sequences is $B_{1/2}$. So we are still in the same setting with Cantor space and computable transformations, and can apply our property.}  independent real numbers $\vec{\xi}= \xi_1,\xi_2,\ldots$ uniformly distributed in $[0,1]$ and then converts them to bits using threshold $1/2$ (i.e., the $i$-th bit is $1$ if $\xi_i<1/2$ and $0$ otherwise). The output distribution is $B_{1/2}$, so by the `no-randomness-from-nothing' property any $B_{1/2}$-random sequence $\omega$ is can be obtained from some random sequence $\vec{\xi}$ of reals. Now consider the same process, but with a threshold of $1/3$. The measure induced by this process is $B_{1/3}$, so by randomness conservation, we get a $B_{1/3}$-random sequence, and this sequence is coordinate-wise smaller than the original one.

Recently a notion of \emph{layerwise computable mapping} was introduced~\cite{mh}; it provides a good level of generality for the image randomness property, and makes the proof nicely balanced. In the following section we reproduce the definition of layerwise computable mapping and discuss its reformulation in terms of machines with rewritable output. Then (Section~\ref{proof}) we give the proof of image randomness property for layerwise computable mappings. Finally, in Section~\ref{conclusion} we consider some extensions and version of the main result (including a quantitative version for randomness deficiencies, and a version for multi-valued mappings).

\section{Layerwise computable mappings}\label{sec:layerwise}

Let $P$ be a computable distribution on the Cantor space $\Omega=\{0,1\}^{\mathbb{N}}$. Computability means that the function $p(x)=P(x\Omega)$, where $x\Omega$ is the set of all infinite sequences that have prefix $x$, is computable: given $x$ and rational $\varepsilon>0$, we can compute some rational $\varepsilon$-approximation to $p(x)$. Let us recall Martin-L\"of's definition of randomness. A \emph{Martin-L\"of test} is a decreasing sequence 
    $$
U_1 \supset U_2 \supset U_3\supset\ldots    
    $$ 
of uniformly effectively open\footnote{This means that there is an algorithm that, given $n$, enumerates a sequence $x_0,x_1,x_2,\ldots$ of strings such that $U_n$ is the union of intervals $x_i\Omega$.} sets such that $P(U_i)\le 2^{-i}$.  A sequence $\omega$ is \emph{rejected} by this test if $\omega\in U_n$ for all $n$. Sequences that are not rejected by any Martin-L\"of test are called \emph{Martin-L\"of random} with respect to the distribution $P$. We also call them \emph{Martin-L\"of $P$-random} or just \emph{$P$-random}.

Consider some mapping $F$ defined on the set of all $P$-random sequences; the values of $F$ are infinite bit sequences. We say that $F$ is \emph{layerwise computable} if there exists a Martin-L\"of test
    $$
U_1 \supset U_2 \supset U_3\supset\ldots    
    $$ 
and a machine $M$ that, given $i$ and a sequence $\omega\notin U_i$ (on the input tape), computes $F(\omega)$, i.e., prints $F(\omega)$ bit by bit on the output tape. Note that every random sequence $\omega$ is outside $U_i$ for large enough $i$, so $F(\omega)$ is computed by $M$ for all sufficiently large $i$. In other terms, for a layerwise computable function $F$ we do not have one algorithm that computes $F(\omega)$ for all $\omega$, but a family of algorithms $M_i(\cdot)=M(i,\cdot)$ that compute $F$ outside the sets $U_i$ of decreasing measure. Note that the behavior of $M_i$ on $\omega\in U_i$ is not specified: $M_i(\omega)$ may be \bl{an} arbitrary finite or infinite sequence (and may differ from $F(\omega)$ even if $F(\omega)$ is defined).

The original definition of layerwise computable mappings used the notion of \emph{randomness deficiency}. Here is a definition of this notion. Given a Martin-L\"of test $U_1\supset U_2\supset\ldots$, we define the corresponding deficiency function 
    $
d_U(\omega)=\sup \{i\mid \omega\in U_i\}.   
    $
The function $d_U$ is infinite on the sequences rejected by $U$, so $d_U(\omega)$ is finite for all random sequences. The definition of layerwise computable mapping can be now reformulated as follows: \emph{there is a machine that computes $F(\omega)$ given $\omega$ and some $i>d_U(\omega)$}.

Martin-L\"of showed that there exists a \emph{universal} test $V$ for which the deficiency function $d_V$ is maximal (up to an additive constant); this test, therefore, rejects \emph{all} non-random sequences. Let us fix some universal test $V$. Since $d_V$ is maximal, an upper bound for $d_V(\omega)$ is also (up to some constant) an upper bound for $d_U$ for every test $U$. So we may say that $F$ is layerwise computable if there is an algorithm that computes $F(\omega)$ given $\omega$ and an upper bound for $d_V(\omega)$. This is how the layerwise computable mappings were defined in~\cite{mh}.
\smallskip

To understand better the intuitive meaning of this definition, let us consider layerwise computable mappings with one-bit output (the value is $0$ or $1$). The definition is essentially the same: a mapping $f$ is layerwise computable if there exists a test $U_i$ and a machine $M$ that, given a sequence $\omega$ and $i$ such that $\omega\notin U_i$, computes~$f(\omega)$. The difference is that now $f(\omega)$ is just one bit, not a bit sequence.

Let $F$ be a mapping defined on random sequences whose values are infinite sequences, and let $F_k$ are individual bits of this mapping, so  
     $$
F(\omega)=F_0(\omega) F_1(\omega) F_2(\omega)\ldots
     $$
The following simple observation says that we can define layerwise computability in a co\-ordinate-wise way:

\begin{proposition}\label{pr:coordinates}
The mapping $F$ is layerwise computable if and only if all the mappings $F_k$ are uniformly layerwise computable.
\end{proposition}
   
\begin{proof}
One direction is obvious: if $F$ is layerwise computable, the same machine can be used for all $k$ to compute $F_k$.  The other direction requires some reasoning: layerwise computability of $F_k$ gives us some test $U^k$, i.e., a decreasing sequence $U^{k}_{1}\supset U^{k}_2\supset U^{k}_3\supset\ldots$ of uniformly effectively open sets, and some algorithm $M^k$ that\bl{,} given $i$, computes $F_k(\omega)$ for all $\omega\notin U^{k}_{i}$. We combine these tests $U^k$ into one test $U$ along the diagonals : let
     $$U_1=U^{0}_2 \cup U^{1}_3 \cup U^{2}_4 \cup\ldots,\quad
       U_2=U^{0}_3 \cup U^{1}_4 \cup U^{2}_5 \cup\ldots,\ \text{etc.}$$
Then we have
    $$P(U_i)\le P(U^{0}_{i+1})+ P(U^{1}_{i+2})+P(U^{2}_{i+3})+\ldots\le 2^{-(i+1)}+2^{-(i+2)}+2^{-(i+3)}+\ldots \le 2^{-i},$$
as required.  We can compute all the bits of $F(\omega)$ for $\omega\notin U_i$, using upper bound $i+k+1$ in the algorithm that layerwise-computes $F_k(\omega)$.  
\end{proof}
   
A layerwise computable mapping $f$ with bit values splits all random $P$-sequences into two classes (preimage of $1$ and preimage of $0$), and therefore defines a point in a metric space of all measurable subsets $X$ of $\Omega$. This metric space has distance function $d(X,Y)=P(X\bigtriangleup Y)$; the points in this space are equivalence classes (two sets $X,Y$ are equivalent if $d(X,Y)=0$, i.e., their symmetric difference $X\bigtriangleup Y$ is a \bl{$P$-}null set). In this space the family of clopen sets (in other terms, the family of sets that are finite unions of intervals) is a countable dense subset on which the distance function is computable in a natural sense (recall that $P$ is a computable distribution). So this space can be considered as a constructive metric space, and computable points in it are defined as points for which we can compute a clopen approximation with any given precision.  An equivalent definition of computable points: consider computable sequences of clopen sets $C_0,C_1,C_2,\ldots$ such that $d(C_i,C_{i+1})\le 2^{-i}$; their limits are computable points.

\begin{proposition}
       \label{pr:points}
Each layerwise computable mapping $f$ of Cantor space to $\{0,1\}$ corresponds to a computable point in the metric space of measureable subsets of $\Omega$. This is a one-to-one correspondence: different layerwise computable mappings correspond to different computable points, and each computable point is obtained from some layerwise computable mapping.
\end{proposition}

\begin{proof}
Let $f$ be a layerwise computable mapping and \bl{let} $U_1\supset U_2\supset\ldots$ be a corresponding test. Let $M_i$ be a machine that computes $f$ outside $U_i$. The domain of $M_i$ (the set of sequences where $M_i$ produces $0$ or $1$) is an effectively open set; simulating $M_i$ on all inputs, we can enumerate intervals of this domain together with the corresponding values ($0$ or~$1$). Denote this domain by $V_i$; by definition we know that $V_i\cup U_i$ is the entire Cantor space. Due to compactness, we can wait until finitely many intervals from $U_i$ and $V_i$ are found that cover (together) the entire space. Let $C_i$ be the union of the intervals in $V_i$ that are found up to this moment with value $1$. Note that $f$ is not guaranteed to be $1$ on all the points in $C_i$, since some of them may belong to $U_i$ and in this case the value of $M_i$ may differ from the value of $f$. But we know that the symmetric difference of $f^{-1}(1)$ and $C_i$ is contained in $U_i$, so it has $P$-measure at most $2^{-i}$. Therefore, the point $f^{-1}(1)$ is computable. Note also that the sets $C_i$ (their indicator functions) converge to $f$ pointwise at every $P$-random point.

To prove the rest of the proposition, let us construct the inverse mapping.  Assume that $C$ is some computable point. Then there exists  a computable sequence $C_0,C_1,C_2,\ldots$ of clopen sets such that $d(C_i,C)\le 2^{-i}$. This inequality implies that $d(C_i,C_{i+1})\le 2\cdot 2^{-i}$ (both $C_i$ and $C_{i+1}$ are close to the limit, so they are close to each other). Now the set
     $$
W_i=(C_i\bigtriangleup C_{i+1})\cup (C_{i+1}\bigtriangleup C_{i+2})\cup
(C_{i+2}\bigtriangleup C_{i+3})\cup\ldots
     $$
is an effectively open set of measure at most $2^{i-1}+2^{i}+2^{i+1}+2^{i+2}+\ldots=4\cdot 2^{-i}$ that covers all the points $\omega$ where the sequence $C_i(\omega)$ does not stabilizes. Therefore, these bad points are all non-random and the limit mapping is well defined on all random points. This limit mapping is layerwise computable, since the sets $W_i$ (more precisely, $W_{i+2}$) form a test, and $C_i$ computes the limit mapping outside $W_i$. 

Now let us show that this correspondence (computable points $\to$ layerwise computable mappings) is well defined, i.e., we get the same layerwise computable mapping starting from two representations of the same point. If $C_0,C_1,\ldots$ and $C'_0,C'_1,\ldots$ are two computable sequences of clopen sets that both converge to some computable point $C$ (with prescribed speed of convergence), then $d(C_i,C'_i)$ is also bounded by $2\cdot 2^{-i}$, and 
     $$
S_i=(C_i\bigtriangleup C'_{i})\cup (C_{i+1}\bigtriangleup C'_{i+1})\cup
(C_{i+2}\bigtriangleup C'_{i+2})\cup\ldots
     $$
is again a test, so for every random sequence $\omega$ we have $C_i(\omega)=C'_i(\omega)$ for all sufficiently large $i$. Therefore, $C_i$ and $C'_i$ give us the same layerwise computable mapping (defined on random sequences).     

It remains to note that these two constructions are inverse. Starting with a layerwise computable mapping $f$, we get a sequence of clopen sets $C_i$ that pointwise converges to $f$, and also converges to the corresponding point in the metric space (the distance between $C_i$ and the limit is at most $2^{-i}$), so it can be used in the second construction, and we get the same mapping on random sequences as its pointwise limit. On the other hand, starting with a computable point, we get a layerwise computable mapping that represents that same point, since the distance between $C_i$ and the pointwise limit of the $C_n$'s is $O(2^{-i})$. This ends the proof of the proposition.
\end{proof}

\begin{remark}
Looking at the construction of a layerwise computable mapping for a given computable sequence $C_i$, we may note that the test $S_i$ appearing in this construction is in fact a \emph{Schnorr test}; this means that the measure of $S_i$ is uniformly computable given~$i$. Therefore we get convergence for all \emph{Schnorr random} sequences $\omega$, i.e., for all sequences not rejected by any Schnorr test. Combining this observation with Proposition~\ref{pr:coordinates}, we see that every layerwise computable mapping (with bit or sequence values) originally defined on Martin-L\"of random sequences, can be naturally extended to the bigger set of Schnorr random sequences.
\end{remark}
   
Recall that in the original statement about image randomness we considered computable mappings of the Cantor space into a space of finite and infinite sequences whose values are $P$-almost everywhere infinite. To see that those mappings are indeed a special case of our notion of a layerwise computable mapping, we need the following simple result.

\begin{proposition}\label{pr:general}
Let $P$ be a computable distribution on the Cantor space. Let $F$ be a computable mapping of the Cantor space to the space of finite and infinite sequences. Assume that $F(\omega)$ is an infinite sequence for $P$-almost all~$\omega$. Then $F(\omega)$ is infinite for all Martin-L\"of random sequences, and $F$ is a layerwise computable mapping on them.
\end{proposition}   

\begin{proof}
According to Proposition~\ref{pr:coordinates}, we may consider each coordinate of the mapping $F(\omega)=F_0(\omega)F_1(\omega)\ldots$ separately. We know that $F_i(\omega)$ is defined for $P$-almost all sequences, so the domain of $F_i$ is an effectively open set of $P$-measure $1$. We enumerate this set until we get a clopen lower bound for it of measure at least $1-2^{-i}$. Its complement $U_i$ is a (cl)open set of $P$-measure at most $2^{-i}$, and we get a Martin-L\"of test.  So every Martin-L\"of random sequence belongs to the domain of $F_i$, and we can compute $F_i$ outside of $U_i$ (in fact, the same machine is used for all $i$). So all $F_i$ (and $F$ in general) are layerwise computable. 
\end{proof}

We finish this section by giving a natural ``machine-language'' description of layerwise computable mapping. Consider a machine $M$ that has a read-only input tape, work tapes, and a \emph{rewritable} output tape. Initially the output tape contains all zeros, but the machine can change them to ones back and forth at any time. What matters is only the limit value of a cell if the limit exists; if not, the corresponding output bit is undefined. Without additional restrictions, this model is equivalent to computability relativized to the halting problem (and this is not interesting for us now). We pose an additional restriction: for each $i$ and for each rational $\varepsilon>0$ one can compute an integer $N(i,\varepsilon)$ such that the probability of the event
\begin{quote}
\emph{machine $M$ with input $\omega$ changes $i$-th output cell after more than $N(i,\varepsilon)$ steps}
\end{quote}
(taken over $P$-distribution on all inputs $\omega$) does not exceed $\varepsilon$.  
This condition implies that the output is well defined (all output values stabilize) for $P$-almost every input $\omega$.

\begin{proposition}\label{pr:machine-definition}
The output sequence is well defined for all Martin-L\"of random input sequences. The resulting mapping is layerwise computable; every layerwise computable mapping can be obtained in this way.
\end{proposition}

\begin{proof}
Again we may use Proposition~\ref{pr:coordinates} and consider each coordinate separately. The contents of $i$-th output cell after $t$ steps is determined by a finite prefix of \bl{the} input sequence $\omega$, so we get some clopen set. Taking $t=N(i,2^{-k})$, we get in this way some clopen set $C_k$ and know that the indicator function of $C_k$ coincides with the limit value for all inputs except for a set of measure at most $2^{-k}$. Then we use the same argument as in Proposition~\ref{pr:points}, and conclude that the resulting mapping is layerwise computable. 

It remains to show that every layerwise computable mapping can be transformed into a machine of described type. As we have seen, a layerwise computable mapping with bit values is a limit of a sequence of (\bl{the} indicator functions of) clopen sets $C_n$ such that $P(C_n\bigtriangleup C_{n+1})<2^{-n}$. The machine $M$ then works in stages; at $n$-th stage it computes $C_n$ and then, using $C_n$ as a table, changes (if necessary) the output bit to make it consistent with $C_n(\omega)$, where $\omega$ is the input sequence. Then $N(2^{-n})$ is the number of steps sufficient to perform $n+1$ stages of this procedure (for arbitrary input $\omega$).

This is for one\bl{-}bit output; we can do the same in parallel for all output bits (interleaving the computations for different output bits) and then define $N(i,2^{-n})$ in such a way that the machine is guaranteed to terminate $n+1$ stages for bit $i$ in at most $N(i,2^{-n})$ steps.
\end{proof}

This characterization of layerwise computable mappings is not used much in the sequel; we provide it here for two reasons. First, it gives some intuitive understanding of this notion on the ``machine language'' level. Second, it turns out that this is not only a theoretical characterization of layerwise computable mappings: some natural randomized algorithms really define a layerwise computable mapping in this way. One example where this phenomenon happens is an infinite version of Moser--Tardos algorithm for Lovasz local lemma used by Andrei Rumyantsev to prove a computable version of this lemma (see~\cite{rumyantsev-arxiv,rumyantsev-exposition}).
   
\section{Image randomness}
\label{proof}

In the section we prove the image randomness property for layerwise computable mappings (Theorem~\ref{th:main}).

Let $P$ be a computable distribution on the Cantor space $\Omega$, and let $F\colon \Omega\to\Omega$ be a layerwise computable mapping defined on $P$-random sequences and having infinite bit sequences as values. Then an image probability distribution $Q$ on the Cantor space is defined in the usual way: $Q(X)=P(F^{-1}(X))$. Note that since $F$ is defined $P$-almost everywere, the distribution $Q$ is well defined.  

\begin{theorem}\label{th:main}

\textup{\textbf{(i)}}~The distribution $Q$ is computable.

\textup{\textbf{(ii)}}~If a sequence $\alpha$ is Martin-L\"of random with respect to $P$, then $F(\alpha)$ is Martin-L\"of random with respect to $Q$.

\textup{\textbf{(iii)}}~If a sequence $\beta$ is Martin-L\"of random with respect to $Q$, then there exists some sequence $\alpha$ that is Martin-L\"of random with respect to $P$ such that $F(\alpha)=\beta$.
\end{theorem}

The first statement (i) is needed to make the notion of $Q$-randomness well defined. The second statement is \bl{the} randomness conservation property, and the third statement says that randomness does not ``appear from nothing''.

\begin{proof}
\textup{\textbf{(i)}} According to the definition, we need to compute the probability of the event ``$F(\omega)$ starts with $u$'' for arbitrary bit string $u$ with arbitrary given precision $\varepsilon>0$. The probability here is taken over $P$-distribution on input sequence $\omega$. Assume that $u$ consists of $k$ bits. Using the machine characterization, we can computably find how many steps are needed to finalize each of $k$ first output bits \bl{with probability of later change at most $\varepsilon/k$}. Then we take \bl{the} maximal of these numbers:
      $$
N=\max(N(0,\varepsilon/k), N(1,\varepsilon/k),\ldots,N(k-1,\varepsilon/k))
     $$
Then the probability in question is $\varepsilon$-close to the same probability for \bl{the} time-bounded computation stopped after $N$ steps, and this probability can be computed by simulating the machine behavior on all possible inputs.

\textup{\textbf{(ii)}} Assuming that $\beta=F(\alpha)$ is not $Q$-random for some $P$-random $\alpha$, we need to get a contradiction. Consider a Martin-L\"of test $U_1\supset U_2\supset\ldots$ provided by the layerwise computability of $F$, and a Martin-L\"of test $V_1\supset V_2\supset\ldots$ that rejects $\beta$. Using them, we construct a test that rejects $\alpha$. By definition, there exists a computable mapping $F_i$ that coincides with $F$ outside $U_i$.  Let us consider $F_i^{-1}(V_i)$. It requires some care since $V_i$ is an open set in the Cantor space of infinite sequences, while the values of $F_i$ may be finite. Still we may represent $V_i$ as a union of intervals $x_{i,0}\Omega\cup x_{i,1}\Omega \cup x_{i,2}\Omega\cup\ldots$, and consider all $\omega$ such that $F_i(\omega)$ has one of the strings $x_{i,0}, x_{i,1},x_{i,2},\ldots$ as a prefix. In this way we get an effectively open set that is denoted $F_i^{-1}(V_i)$ in the sequel. This notation is formally incorrect since the set depends not only on $V_i$, but also on its representation as the union of intervals, but this will not create a problem.

The set $F_i^{-1}(V_i)$ consists of two parts: sequences $\omega\in U_i$, and sequences $\omega\notin U_i$. The first part is small since $U_i$ is small; the second part is small since $F_i$ coincides with $F$ outside $U_i$, so the second part is contained in $F^{-1}(V_i)$. So the measures of both parts are bounded by $2^{-i}$, and we get at most $2\cdot 2^{-i}$ in total. Now we add $U_i$ to $F_i^{-1}(V_i)$ and get an effectively open set of measure $O(2^{-i})$ that covers $\alpha$ both for the case $\alpha\in U_i$ and for the case $\alpha\notin U_i$. Since this can be done for arbitrary $i$, the sequence $\alpha$ is not random.

\textup{\textbf{(iii)}} Let $\beta$ be a $Q$-random sequence. We want to show that $\beta$ belongs to the image $F(R)$ where $R$ is the set of all $P$-random sequences. Martin-L\"of's definition of randomness guarantees that the complement of $R$ can be represented as the intersection of a decreasing sequence of effectively open sets from the universal Martin-L\"of test, so $R$ can be presented as \bl{the} union of an increasing sequence of their complements, i.e., a sequence 
$
R_1\subset R_2 \subset R_3\subset\ldots,
$
of effectively closed sets where $P$-measure of $R_i$ is at least $1-2^{-i}$. The mapping $F$ is defined everywhere on $R$, so it is defined everywhere on $R_i$. Moreover, since we started from the \emph{universal} Martin-L\"of test, the definition of layerwise computability guarantees that on $R_i$ the mapping $F$ may be computed by a machine with \bl{write-only output tape,} so $F$ is continuous on $R_i$. The set $R_i$ is compact (being a closed set in a compact space), so its image $F(R_i)$ is compact and therefore closed. The $Q$-measure of $F(R_i)$ is at least $1-2^{-i}$, since the preimage $F^{-1}(F(R_i))$ of $F(R_i)$ contains $R_i$. 

It remains to show that the sets $F(R_i)$ are uniformly effectively closed (i.e., their complements are uniformly effectively open). Indeed, in this case their intersection contains only non-$Q$-random sequences, so the sequence $\beta$ belongs to some $F(R_i)$ and is therefore an image of a (random) sequence from $R_i$. 

For that we may use the effective version of an argument proving that an image of a compact set is compact. We enumerate intervals $V$ that are disjoint with $F(R_i)$ in such a way that their union is the complement of $F(R_i)$. For that, we look for a tuple $(U_1,\ldots, U_k)$ of  intervals and an interval $V$ such that (a)~$V$ is disjoint with all $U_i$; 
 (b)~ $F^{-1}(U_1), F^{-1}(U_2),\ldots, F^{-1}(U_k)$
cover $R_i$. The conditions (a) and (b) guarantee that $V$ is disjoint with $F(R_i)$: the image of every element in $R_i$ belongs to one of the $U_i$ and therefore does not belong to $V$. Note that we can replace $F$ by $F_i$ in (b) and get an equivalent condition, since $F=F_i$ on $R_i$, so a cover remains a cover.  Using this observation, we can enumerate all the tuples $U_1,\ldots,U_k,V$ (since the complement of $R_i$ and $F_i^{-1}(U_1),\ldots,F^{-1}_i(U_k)$ are effectively open and $\Omega$ is compact, we can look for cases when some intervals inside the preimages together with some intervals inside the complement of $R_i$ cover the entire $\Omega$). When such a tuple $U_1,\ldots,U_k,V$ is found, we include the interval $V$ in the enumeration. It does not intersect $F(R_i)$, as we have seen; the standard compactness argument shows that every point outside $F(R_i)$ is covered by one of those intervals $V$, so $F(R_i)$ indeed has an effectively open complement. This construction is uniform in $i$.
\end{proof}

\section{Remarks and enhancements}
\label{conclusion}

\subsection{Image randomness and van Lambalgen theorem}

Every sequence $\alpha=\alpha_0\alpha_1\alpha_2\alpha_3\ldots$ can be split into two sequences
$\alpha^\mathrm{e}=\alpha_0\alpha_2\ldots$ and 
$\alpha^\mathrm{o}=\alpha_1\alpha_3\ldots$
by separating its even and odd terms. The mapping $\alpha\mapsto (\alpha^\mathrm{e},\alpha^\mathrm{o})$ provides an isomorphism between the Cantor space $\Omega$ and its square $\Omega\times\Omega$, where $\Omega$ is considered as probability space equipped with uniform Bernoulli distribution.  Applying Theorem~\ref{th:main} to the function $\alpha\mapsto\alpha^\mathrm{e}$ (the composition of our isomorphism and the projection function), we see that
\begin{itemize}
\item if $\alpha$ is random, then $\alpha^\mathrm{e}$ is random;
\item every random sequence $\beta$ is equal to $\alpha^\mathrm{e}$ for some random $\alpha$.
\end{itemize}
The second statement means that for every random sequence $\alpha^\mathrm{e}$ we can also find some $\alpha^\mathrm{o}$ in such a way that together they make some random sequence $\alpha$. One may ask which sequences $\alpha^\mathrm{o}$ can be used to achieve this. The answer is given by the well known \emph{van Lambalgen theorem} that says that $\alpha$ is random if and only if $\alpha^\mathrm{e}$ is random and $\alpha^\mathrm{o}$ is random \emph{with respect to the oracle $\alpha^\mathrm{e}$}. The oracle randomness is defined in the usual way: the sets of small measure that form the test should be effectively open with the oracle (the enumerating algorithm may consult the oracle).

In general, we have two computable distributions $P_1$ and $P_2$ on two copies of $\Omega$ and consider randomness with respect to their product $P_1\times P_2$. The van Lambalgen theorem says that \bl{a} pair $(\alpha_1,\alpha_2)$ is $(P_1\times P_2)$-random if and only  if $\alpha_1$ is $P_1$-random and $\alpha_2$ is $P_2$-random with respect to oracle $\alpha_1$.  It implies that for every $P_1$-random sequence $\alpha_1$ there exists some $\alpha_2$ that makes the pair $(\alpha_1,\alpha_2)$ random with respect to $P_1\times P_2$. This (much weaker) statement is also a direct corollary of Theorem~\ref{th:main}. The advantage of using Theorem~\ref{th:main} is that it can be applied to every computable distribution $P$ on $\Omega\times\Omega$, not necessarily the product distribution: if a pair $(\alpha_1,\alpha_2)$ is $P$-random, then $\alpha_1$ is random with respect to the marginal distribution $P_1$; if $\alpha_1$ is $P_1$-random, then $(\alpha_1,\alpha_2)$ is  $P$-random \emph{for some $\alpha_2$}.

One can ask for which $\alpha_2$ this happens, i.e., how one can generalize van Lambalgen's theorem to non-product distributions. There are some results in this direction~\cite{vovk,takahashi}, but we do not go into details here; let us only note that this happens for almost all $\alpha_2$ with respect to the conditional distribution defined according to Takahashi~\cite{takahashi}.

\subsection{Quantitative version}

As many other results about algorithmic randomness, Theorem~\ref{th:main} has a quantitative version. Recall that \emph{randomness deficiency} of a sequence $\omega$ may be defined as the maximal number $i$ such that $\omega\in U_i$, where $U_1\supset U_2\supset U_2\supset \ldots$ is \bl{the} universal Martin-L\"of test {(see~\cite{gacs,usv}; this version of randomness deficiency is called \emph{probability-bounded} deficiency there). Using this notion, we can state the promised quantitative version.

Let $P$ be a computable distribution on $\Omega$, let $F$ be an $\Omega$-valued layerwise computable mapping defined on $P$-random sequences, and let $Q$ be the $F$-image of $P$.

\begin{theorem}\label{th:probability-bounded-def}
$$d_Q(\beta)=\inf \{ d_P(\alpha) \mid F(\alpha)=\beta\}+O(1).$$
\end{theorem}

Here $d_P$ and $d_Q$ stand for the randomness deficiency functions for distributions $P$ and $Q$; the constant in $O(1)$ depends on the choice of deficiency functions and on $F$. Note that this result directly implies Theorem~\ref{th:main} (recall that the infimum of \bl{the} empty set is $+\infty$).

\begin{proof}
For the proof we should just look closely at the argument we used. First, we have shown  by contraposition that an image $F(\alpha)$ for $P$-random $\alpha$ is $Q$-random. Now we need to show that if the deficiency of $F(\alpha)$ is high, then the deficiency of $\alpha$ is high. Indeed, if $\beta=F(\alpha)$ is covered by $V_i$ (in the notation used to prove Theorem~\ref{th:main}), then $\alpha$ is covered by the set $W_i=U_i\cup F_i^{-1}(V_i)$ whose measure is $O(2^{-i})$. The sets $W_i$ also form a Martin-L\"of test (up to a $O(1)$-change in numbering), so 
$$
d_Q(F(\alpha))\le d_P(\alpha)+O(1)
$$
(the quantitative version of randomness conservation).

To prove the reverse inequality, note that the complements of the sets $F(R_i)$ form a $Q$-test. So if $d_Q(\beta)$ is small and $i$ is large, the sequence $\beta$ belongs to $F(R_i)$ and not to its complement. More formally, $\beta$ belongs to $F(R_i)$ for $i$ that is only $O(1)$-larger than $d_Q(\beta)$, and all elements in $F(R_i)$ have preimages in $R_i$ whose deficiency is bounded by~$i$.
\end{proof}

\subsection{Quantitative version for expectation-bounded deficiency}

There is another version of randomness deficiency introduced by Levin. This version is called in~\cite{gacs,usv} the \emph{expectation-bounded} deficiency,  while the original Martin-L\"of version (described above) is called \emph{probability-bounded}. In this version $d_P(\cdot)$ is defined as the logarithm of the maximal (up to a $O(1)$-factor) non-negative lower semicomputable function with finite integral with respect to a given distribution (see~\cite{gacs,usv} for details).  These two versions of randomness deficiency differ more than by a constant (the probability-bounded one is bigger, and the difference is unbounded), but the formula of Theorem~\ref{th:probability-bounded-def} is true also for the expectation-bounded deficiency. Let us prove this (assuming that the reader is familar with the basic techniques from~\cite{gacs}).

Let us show first that the randomness conservation property $d_Q(F(\alpha))\le d_P(\alpha)+O(1)$ for this version of the randomness deficiency. Note that if $\alpha$ is not random, $F(\alpha)$ is undefined and $d_P(\alpha)$ is infinite, so the claim makes sense only if $\alpha$ is $P$-random, and $F(\alpha)$ is defined for every $P$-random $\alpha$. Let $t_Q(\cdot)$ be the maximal lower semicomputable integrable function (the universal randomness test) for distribution $Q$, so $d_Q(\cdot)=\log t_Q(\cdot)$. We assume that $\int t_Q(\beta)\,dQ(\beta)\le 1$. We may try to consider a function $t(\alpha)=t_Q(F(\alpha))$ and show that it is a $P$-test and therefore is bounded by the universal $P$-test (so we get the inequality we want to prove). However, this plan does not work ``as is'', because the function $t(\alpha)$ is well defined only for random $\alpha$ and there is no reason to expect that $t$ can be extended to a lower semicomputable function on the entire $\Omega$ (though the other requirement for tests, bounded integral, is true for~$t$).

To overcome this difficulty, we recall that $F$ is a layerwise computable mapping. We may assume that $F$ coincides with computable $F_i$ outside some effectively open set $U_i$ whose $P$-measure is at most $2^{-i}$. Now we may consider a function
$$
t_i(\alpha)=t_Q(F_i(\alpha))
$$
defined on random $\alpha$ and, moreover, on all $\alpha$ where $F_i(\alpha)$ is infinite.  It is easy to see that $t_i$ can be naturally extended to a lower semicomputable function defined on the entire $\Omega$. We denote this extension also by~$t_i$. Now the problem is that  there is no reason to expect that $\int_{\Omega} t_i(\alpha)\, dP(\alpha)$ is finite; the only thing we are sure about is that $\int_{\Omega\setminus U_i} t_i(\alpha)\, dP(\alpha)\le1.$ This happens because $t_i=t$ outside $U_i$, and for $t$ the similar inequality is true (recall that $Q$ is the image of $P$).

So we take \bl{the} next step and consider the function $t_i'$ defined as $t_{2i}$ artificially cut at $2^i$ (all bigger values are replaced by $2^i$). It is still lower semicomputable, but in this way we guarantee that the integral of $t_i'$ inside $U_{2i}$ does not exceed $2^{-2i}\times 2^i=2^{-i}$. Let $t'(\alpha)=\sup_i t_i'(\alpha)$. The function $t'$ is lower semicomputable since all $t_i'$ are uniformly lower semicomputable. To get the bound for the integral, let us consider $t^N(\alpha)=\sup_{i\le N} t_i'(\alpha)$. When can $t^N(\alpha)$ exceed $t(\alpha)$ for a random $\alpha$? If this happens because of some $t_i'(\alpha)$, then $\alpha$ belongs to $U_{2i}$, and the excess is bounded by a function that equals $2^i$ in $U_{2i}$ and equals $0$ outside $U_{2i}$. The integral of this function is at most $2^{-i}$, and the total excess is bounded by the sum of these integrals, so it is bounded by $1$. We conclude that $\int t'(\alpha)\,dP(\alpha)\le 2$ (since $\int t(\alpha)\,dP(\alpha)\le 1$ and the excess is bounded by $1$), so $\log t'(\alpha)\le d_P(\alpha)+O(1)$. It remains to note that $t(\alpha)=t_Q(F(\alpha))\le t'(\alpha)$ for random $\alpha$, since $\alpha\notin U_{2i}$ for large $i$, and for those $i$ we have $t_i'(\alpha)=\min(t(\alpha),2^i)$, so the supremum over $i$ is at least $t(\alpha)$.
\smallskip

Now the reverse inequality. Consider the universal test $t_P(\cdot)$ for distribution $P$, i.e., a maximal non-negative $P$-integrable lower semicomputable function; its logarithm is $d_P(\cdot)$. Now consider the function
      $$u(\beta)=\inf\{t_P(\alpha) \mid F(\alpha)=\beta\},$$
where the infimum is taken over random $\alpha$ (otherwise $F(\alpha)$ is not defined). If $\alpha$ is \bl{a} $P$-distributed random variable, then $\beta=F(\alpha)$ is \bl{a} $Q$-distributed random variable, and $u(F(\alpha))\le t_P(\alpha)$, so 
     $$\int u(\beta)\,dQ(\beta)=\int u(F(\alpha))\,dP(\alpha)\le \int t_P(\alpha)\,dP(\alpha) \le 1$$
(assuming that $u$ is measurable, otherwise the integral may be not well defined).
     
It remains to prove that $u$ is lower semicomputable. Indeed, in this case the integral we consider is well defined, the function $u$ is a $Q$-test, so $\log u(\beta)\le d_Q(\beta)$, and that is exactly what we need. To prove the lower semicomputability of $u$ let us recall the definition. We need to prove that the set $\{\beta \mid u(\beta)>r\}$ is uniformly effectively open for every rational $r$. Note that for this set all the values of $u$ greater than $r$ are equivalent, so it remains unchanged of we let $t_P$ be infinite on $U_i$ for some very large $i$ (where $t_P(\alpha)$ is large anyway). Here $U_i$ is the effectively open set of small measure provided by the definition of layerwise computability. After that we may change $F$ on $U_i$ making it a total computable mapping, i.e., a mapping that outputs infinite sequences on all inputs (as soon as the input of $F$ is covered by $U_i$, we extend the output arbitrarily, e.g., by zeros). Then we may apply the standard reasoning for the case of a total computable mapping; here it~is.

We have $u(\beta)=\inf\{t_P(\alpha) \mid F(\alpha)=\beta\}>r$ if and only if there exists some rational $s>r$ such that $t_P(\alpha)>s$ for all $\alpha$ such that $F(\alpha)=\beta$. We assume now that $F$ is total (as discussed above), and consider all $\alpha$, not only random ones, since $t_P(\alpha)$ is infinite for all non-random~$\alpha$. The condition ``$t_P(\alpha)>s$ for all $\alpha$ such that $F(\alpha)=\beta$'' means that $\beta$ does not belong to $F(\{\alpha\mid t_P(\alpha)\le s\})$. The latter set is an image of an effectively closed set under a total computable mapping, so (as we have seen in the proof of Theorem~\ref{th:main}) it is an effectively closed set, and its complement is effectively open. The construction is uniform, so we get what we wanted.

\subsection{An application to random closed sets}

Now we want to present an example where the generalization to layerwise computable mappings is important (i.e., the result for computable mappings is not enough). Let us consider two definitions of randomness for trees and show that they are equivalent.
\smallskip

Consider the following two random processes:

\begin{enumerate}

\item Take the full binary tree. Cut independently each edge with probability $1/3$. Then consider the connected component of the root in the remaining graph. It is a binary tree that may be finite or infinite (as we will see, the first case happens with probability $1/4$).  Delete all vertices that have finitely many descendants. If the connected component of the root was finite, this gives an empty tree, otherwise we get an infinite tree without leaves.

\item Consider three possible elements
\begin{center}
\includegraphics{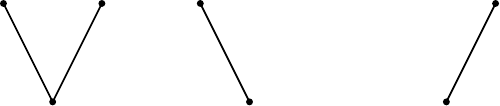}
\end{center}
Select one of them randomly (all three have probability $1/3$) and place it at the root of the tree. Then at every vertex of level $1$ (there can be one or two, depending on the root choice) select one of three elements independently (using fresh randomness), and so on.

\end{enumerate}

Both random processes define a probability distribution on the set of infinite subtrees of binary tree that have no leaves; in the first case we should consider the conditional distribution under the condition ``process generates a non-empty tree''. As we will see, this multiplies all probabilities by $4/3$, since the probability to get an infinite tree is $3/4$.

It turns out that these two distributions on trees without leaves are the same. Moreover, an algorithmic version of the same statement is also true. Imagine that we use bits from a Martin-L\"of random sequence when the random process asks for fresh randomness. It turns out that we get the same class of trees (called Martin-L\"of random trees) in both cases. This was proven in~\cite{diamondstone2012} by some ad hoc construction; in this section we explain why this follows from Theorem~\ref{th:main}.

Let us first compare the distributions. First of all let us explain how the probabilities $1/4$ and $3/4$ (for finite and infinite trees) are obtained. Let $p_n$ be a probability to have a tree of height at least $n$ when cutting each edge of a full binary tree independently with probability $1/3$. The inductive definition of the process gives the recurrence:
$$
p_{n+1}=2q(1-q)p_n+q^2(2p_n-p_n^2)=-\tfrac{4}{9}p_n^2+\tfrac{4}{3}p_n
$$
where $q=2/3$ is the probability that a given edge exists (is not cut). The first term corresponds to two disjoint cases when one of the root's children exists and the other one is cut, the second term is the probability that both children exists and at least one of them has subtree of height at least $n$ (these events have probability $p_n$ and are independent). We start with $p_0=1$ and iterate the function $x\mapsto -\frac{4}{9}x^2+\frac{4}{3}x$; it is easy to see that these iterations converge to the fixed point $3/4$ exponentially fast. So the event ``tree is infinite after edges are randomly cut with probability $1/3$'', being the intersection of a decreasing sequence of events with probabilities $p_n$, has probability $3/4$.

Knowing this, we may compute the probabilities of three possibilities shown above. The first case (where the root has two sons in the final tree) has probability $(2/3)^2(3/4)^2=1/4$, the second and the third cases have probability $$(2/3)(3/4)(1-(2/3)(3/4))=1/4$$ each. So all three cases have the same probability $1/4$ and conditional probability $1/3$, and the same is true at every vertex. So two distributions described above are the same.

Now we switch to algorithmic randomness. First of all, we need to fix the representation of trees by bit sequences. A tree without leaves can be considered as a set of vertices of a full binary tree (that contains with every vertex $x$ its father and one of its sons). Let $\{v_0,v_1,v_2,\ldots\}$ be the vertices of the full binary tree (=binary strings). Then each tree without leaves can be represented as a sequence of bits: $i$th bit is $1$ if $v_i$ belongs to the tree. 

We start with the second random process (in the list above). It uses uniform random choice between three possibilities. So we consider ternary sequences (Martin-L\"of randomness can be defined for any finite alphabet in the same way, and all the results can be easily extended to this case), and \bl{the} uniform distribution on $\{0,1,2\}^\infty$ where each term is uniformly distributed (each of three possible values has probability $1/3$) and the terms are independent. Then for each sequence $\omega\in\{0,1,2\}$ we construct a tree without leaves as described (using terms of the sequence to make choice between three extensions). In this way we get a total computable mapping $\{0,1,2\}^\infty\to\Omega$. Then we consider the image of the uniform distribution on $\{0,1,2\}^\infty$ and get the image distribution on sequences representing trees (and therefore on trees). Random sequences with respect to this image distribution are exactly the images of Martin-L\"of random sequences in $\{0,1,2\}^\infty$ due to Theorem~\ref{th:main}; corresponding trees are called (Martin-L\"of) random trees. 

In fact, this description can be understood in two ways. We may use the terms of a ternary sequence to make choices in all vertices in advance (and then some choices do not matter since the vertex does not appear in the tree). Or we may use the terms of a ternary sequence only for the choices that really matter (starting from the root and then adding layers one by one); in the latter case the same term of the sequence may be used in different vertices depending on the preceding terms. These two possibilities give different mappings $\{0,1,2\}^\infty\to\Omega$. However, standard probability theory arguments show that the image distributions are the same. As we know, this implies that the image of the set of random ternary sequences is the same.

The first random process (generating trees by cutting edges) is more complicated, and layerwise computable mappings appear. We start with a binary sequence (generated according to $B_{1/3}$). Then we cut the edges according to the terms of this binary sequence, and get a tree without leaves or the empty set of vertices (that corresponds to a zero sequence when we encode trees by binary sequences). In this way we get a mapping $\Omega\to\Omega$. 

This mapping is layerwise computable. To prove this, consider a rewriting machine that constructs a tree using random input bits to cut the edges of the full binary tree. This machine starts from the root and goes upwards, but when a vertex $x$ with a finite subtree above $x$ is discovered, this subtree (including $x$ itself) is deleted. So there is always a possibility that some part of the tree, or the entire tree, will be deleted in the future. However, the probability of the event ``the root survived up to height $n$, but later was deleted'' is bounded by $p_n-\lim p_n$ (in fact, is equal to it) and computably converges to zero. Similar statement is true for every vertex. So the mapping is layerwise computable.

What is the image distribution (of $B_{1/3}$)? As we have seen, it has an atom (zero sequence, or empty tree) that has probability $1/4$, and the rest is the previous distribution on trees multiplied by $3/4$. It is easy to see that random sequences are the same (plus a zero sequence) as in the first case. It remains to apply Theorem~\ref{th:main}.

\subsection{Many-valued mappings}

We considered a random process that consists of two parts: a source of random bits (with some distribution $P$) that generates some bit sequence $\alpha$, and some machine $M$ that transforms these bits (the sequence $\alpha$) into an output sequence $\beta$. Now we want to consider a more general situation when the transformation is non-deterministic. We consider $M$ as a black box, and the only thing we know is some relation between its input $\alpha$ and output $\beta$. Formally, we have some closed subset $F$ of $\Omega\times\Omega$, and we know that $\langle\alpha,\beta\rangle$ always belongs to $F$.

The question remains the same: which sequences are plausible as outputs of such a device? For which $\beta$ there exists some $P$-random sequence $\alpha$ such that $\langle \alpha,\beta\rangle\in F$? 

Let us give two motivating examples for this scheme. First, consider a bit sequence that is Martin-L\"of random with respect to $B_p$ (independent bits, each bit equals $1$ with probability $p$). Then replace some ones by zeros. Which sequences can be obtained by this two-stage process? These sequences (correspoding sets) were called \emph{$p$-sparse} in~\cite{2008-sparse}. In our language, we consider $B_p$ as $P$ (source distribution), and coordinate-wise ordering as $F$: $$F(\alpha_0\alpha_1\ldots,\beta_0\beta_1\ldots)\Leftrightarrow \forall i\, (\beta_i\le \alpha_i).$$

Another example: let $\alpha$ be a random tree (in the sense described above) and let $F(\alpha,\beta)$ mean that $\beta$ is a path in $\alpha$. Then we get the question: which sequences may appear as paths in random trees? This question was studied in \cite{cenzeretal,diamondstone2012,miller-rute}; finally these sequences were characterized in terms of continuous a priori probability. 

Let us return to the general question mentioned above, for arbitrary measure $P$ and closed set $F\subset\Omega\times\Omega$. Some answer can be provided in terms of \emph{class randomness} (see~\cite{gacs} for the definition and the basic properties; we assume that the reader is familiar with this notion). Consider some effectively closed set $\mathcal{L}$ in the effective metric space of all distributions on the Cantor space. An \emph{expectation-bounded class test with respect to $\mathcal{L}$} is a lower semicomputable function $t(\cdot)$ with non-negative values defined on the Cantor space such that the integral of $t$ \emph{with respect to any distribution in $\mathcal{L}$} does not exceed~$1$.   For every effectively closed class $\mathcal{L}$ there exists a maximal (up to a constant factor) test $t_\mathcal{L}$. The sequences $\omega$ such that $t_\mathcal{L}(\omega)$ is finite are called \emph{random with respect to class $\mathcal{L}$}, and $d_\mathcal{L}(\omega)=\log t_\mathcal{L}(\omega)$ is called the \emph{expectation-bounded randomness deficiency of $\omega$ with respect to $\mathcal{L}$}. If $\mathcal{L}$ is a singleton, its only element is a computable distribution and $\mathcal{L}$-randomness is equivalent to Martin-L\"of randomness. In the general case, a sequence $\omega$ is $\mathcal{L}$-random if and only if it is uniformly random with respect to some distribution $Q\in\mathcal{L}$ (this distribution may not be computable, so a notion of randomness with respect to non-computable distributions is needed, see~\cite{gacs}). 

The notion of class randomness can be equivalently defined in terms of Martin-L\"of class tests: a decreasing sequence $U_1\supset U_2\supset\ldots$ is a $\mathcal{L}$-test if $Q(U_i)\le 2^{-i}$ for every $Q\in\mathcal{L}$. For each test a deficiency function is introduced whose value on $\omega$ is $\sup\{i\mid \omega\in U_i\}$; there exist universal tests for which the deficiency function is maximal, and this maximal function is called the \emph{probability-bounded randomness deficiency with respect to $\mathcal{L}$}. (We can also use probability-bounded tests defined as lower semicomputable functions $t$ such that for every $c$ the event $t(\omega)>c$ has probability at most $1/c$ with respect to every distribution in $\mathcal{L}$; this gives an equivalent definition of probability-bounded deficiency.)

Now we can formulate the promised answer about sequences $\beta$ such that $F(\alpha,\beta)$ for some $P$-random $\alpha$, assuming that the relation $F$ satisfies some conditions. Namely, we assume that 

\begin{itemize}
\item[(i)] $F$ is an effectively closed subset of $\Omega\times\Omega$.
\item[(ii)] The projection of $F$ on the first coordinate is $\Omega$: for every $\alpha$ there exists some $\beta$ such that $F(\alpha,\beta)$.
\item[(iii)] The $F$-preimage of every interval $[y]$ in the Cantor space, i.e., the set of all $\alpha$ such that $F(\alpha,\beta)$ for some $\beta$ starting with $y$, is a clopen set (a finite union of intervals) that can be computed given $y$.
\end{itemize}
It is easy to check that all these conditions are satisfied in the two examples given above.

Recall that we consider some computable distribution $P$ on $\Omega$ (the first coordinate of the product $\Omega\times\Omega$). Consider the class $\mathcal{L}$ of all distributions $Q$ on the second copy of $\Omega$ that are \emph{$F$-couplable with $P$}. This means, by definition, that there exists some distribution $S$ on $\Omega\times\Omega$ concentrated inside $F$ (the complement of $F$ has $S$-measure $0$) such that the first projection (marginal distribution) of $S$ is $P$ and the second projection is $Q$.

\begin{proposition}\label{prop:non-empty}
This class $\mathcal{L}$ is a non-empty effectively closed class of distributions on~$\Omega$.
\end{proposition}

\begin{proof}
The condition``the first projection of $S$ equals $P$'' is effectively closed (defines an effectively closed class of distributions $S$). The same can be said about the condition ``$S$ is concentrated inside $F$'' (the complement of $F$ is a union of a computable sequence of intervals in $\Omega\times\Omega$, and we require the $S$-measure of each of these intervals to be $0$, thus getting a sequence of uniformly effectively closed conditions). Let us check that these two conditions have non-empty intersection. By compactness, it is enough to consider some clopen $F'\supset F$ instead of $F$: if for each $F'$ the intersection is non-empty, then it is non-empty for $F$. The projection of $F'$ contains the projection of $F$ and therefore covers the entire $\Omega$; since $F'$ is clopen (a finite union of intervals), we can represent $\Omega$ as a disjoint union $\Omega=U_1+\ldots+U_n$ of intervals, and find some intervals $V_1,\ldots,V_n$ such that $U_i\times V_i\subset F'$ for every $i$. Then on each $U_i\times V_i$ we consider a distribution that is a product of $P$ restricted to $U_i$, and some computable distribution on $V_i$, and combine these products to get a distribution in the desired class (projection is $P$ and support is in $F'$).

It remains to note that this intersection is an effectively closed (and therefore compact) class of distributions on $\Omega\times\Omega$; the projection to the second coordinate is a computable continuous mapping, so its image is again an effectively closed non-empty class of distributions.
\end{proof}

This proposition allows us to consider the randomness deficiency function $d_\mathcal{L}$ for the class $\mathcal{L}$, and compare it with the randomness deficiency function $d_P$ for distribution $P$.

\begin{theorem}
Let $P$ be a computable probability distribution; let $F$ be a subset of $\Omega \times \Omega$ satisfying the conditions \textup{(i)--(iii)}, and let $\mathcal{L}$ be the class of distributions $Q$ which are $F$-couplable with~$P$. The following equality holds with $O(1)$-precision \textup(both for expectation- and probability-bounded versions of deficiency function\textup{):}
$$
d_\mathcal{L}(\beta)=\inf_\alpha \{ d_P(\alpha) \mid F(\alpha,\beta)\},
$$
\end{theorem}

In particular, a sequence $\beta$ is random with respect to the class $\mathcal{L}$ if and only if $F(\alpha,\beta)$ is true for some $P$-random sequence $\alpha$.

\begin{proof}
Let $t_P$ be the maximal expectation-bounded randomness test with respect to $P$, i.e., maximal lower semicomputable function such that $\int t_P(\omega)\,dP(\omega)\le 1$. Then we may consider the function 
        $$t(\beta)=\inf\{t_P(\alpha)\mid F(\alpha,\beta)\}.$$ 
It is easy to check (see~\cite{gacs}) that $t(\cdot)$ is lower semicomputable. Let us check that 
        $$\int t(\beta)\,dQ(\beta)\le 1\text{  for every distribution }Q\in\mathcal{L}.$$
 According to the definition of $\mathcal{L}$, for every $Q\in\mathcal{L}$ there exists some distribution $S$ on $\Omega\times\Omega$ that is zero outside $F$ and has projections $P$ and $Q$ (for the first and second coordinates). Then the integral in question is equal to $\int_{\Omega\times\Omega} t(\beta)\,dS(\alpha,\beta)$. Since $S$ vanishes outside $F$ and $t(\beta)\le t_P(\alpha)$ when $(\alpha,\beta)\in F$,  the latter integral does not exceed $\int_{\Omega\times\Omega}t_P(\alpha)\,dS(\alpha,\beta)=\int t_P(\alpha)\,dP(\alpha)\le 1$. So the function $t(\beta)$ is a test with respect to $\mathcal{L}$ and does not exceed $t_{\mathcal{L}}(\beta)$. One inequality of the theorem ($\ge$) is proven for expectation-bounded tests.

To prove the same inequality for probability-bounded tests, it is enough to show that if $t_P(\cdot)$ is a probabiity-bounded test with respect to $P$, then $t(\cdot)$ is a probability-bounded test with respect to $\mathcal{L}$. Indeed, for a given threshold $c$ consider the set $U=\{\beta\mid t(\beta)>c\}$. We need to prove that $Q(U)\le 1/c$ for every $Q\in\mathcal{L}$. Let $S$ be the corresponding measure on the product space. Then $Q(U)=S(\Omega\times U)=S((\Omega\times U) \cap F)$. The definition of $t(\cdot)$ guarantees that $\Omega\times U\cap F \subset V\times\Omega$, where $V=\{\alpha\mid t_P(\alpha)>c\}$ (if the infimum exceeds $c$, then all the elements exceed $c$). And $S(V\times\Omega)=P(V)\le 1/c$ since $t_P$ is a probability-bounded test.

So one direction is proven for both types of tests. The other direction is more difficult: here we have to use the assumptions about $F$. Again we consider expectation-bounded tests first. We need to show that
$$
  t_{\mathcal{L}}(\beta)\le \inf_\alpha \{t_P(\alpha)\mid F(\alpha,\beta)\}
$$
(for maximal tests, up to a constant factor). This means that 
$$
t_{\mathcal{L}}(\beta)\le t_P(\alpha)\quad\text{if}\quad F(\alpha,\beta).
$$ 
This can be rewritten also as 
$$
  \sup_\beta \{ t_{\mathcal{L}}(\beta)\mid F(\alpha,\beta)\}\le t_P(\alpha).
$$
It would be guaranteed if the left-hand side, denoted by $t(\alpha)$ in the sequel, were an expectation-bounded $P$-test, i.e., were lower semicomputable and had a bounded integral.

Imagine for a while that the supremum is achieved for some measurable function
   $$\beta(\alpha)=\argmax\{t_\mathcal{L}(\beta)\mid F(\alpha,\beta)\},$$ 
so $t(\alpha)=t_{\mathcal{L}}(\beta(\alpha))$.  By construction, the pair $(\alpha,\beta(\alpha))$ belongs to $F$ for every $\alpha$, so the random pair $(\alpha,\beta(\alpha))$  defines a distribution on $\Omega\times\Omega$ whose support is in $F$ and whose first projection is $P$. Then its second projection $Q$ is in $\mathcal{L}$, so the integral $\int t_{\mathcal{L}}(\beta)\, dQ(\beta)$ does not exceed $1$. On the other hand,
   $$\int t_{\mathcal{L}}(\beta)\, dQ(\beta)=\int t_{\mathcal{L}}(\beta(\alpha))\,dP(\alpha) = \int t(\alpha)\, dP(\alpha),$$
so we get the required bound for the function $t$.

Of course, this argument is not valid: the maximum may be never achieved, so the function $\beta(\alpha)$ is not defined; another problem is that the function $t$ may not be lower semicomputable (it can be guaranteed for  the infimum, but not for the supremum). So we need to replace this argument by a valid one.

We start in the same way and consider the maximal lower semicomputable $\mathcal{L}$-test $t_{\mathcal{L}}$. Let $t_\mathcal{L}^1\le t_\mathcal{L}^2\le\ldots$ be a computable sequence of basic functions (functions with finitely many values having clopen preimages) that converges to $t_{\mathcal{L}}$. Now we may consider the functions
         $$t^i(\alpha)=\max\{t_\mathcal{L}^i(\beta)\mid F(\alpha,\beta)\}$$
they are well defined since $t_{\mathcal{L}}^i$ has only finitely many values. Moreover, the conditions on $F$ guarantee that $t^i(\cdot)$ is a basic function that can be computed given $i$. The sequence $t^i$ increases, so we may consider its supremum, some lower semicomputable function $t$. We will show (see below) that $\int t^i(\alpha)\,dP(\alpha)\le 1$ for every $i$. Then the same is true for~$t$, so $t(\alpha)\le c t_P(\alpha)$ for some $c$ and for all $\alpha$ (due to the maximality of $t_P$). Then, for $(\alpha,\beta)\in F$ we have
$$
 t^i_{\mathcal{L}}(\beta) \le t^i(\alpha) \le t(\alpha) \le c t_P(\alpha),
$$
and it remains to take infimum over all $\alpha$ and then supremum over $i$ to get the reverse inequality we need.

So it remains to prove that $\int t^i(\alpha)\,dP(\alpha)\le 1$ for every $i$. This argument follows the informal explanation above, but instead of the function $\beta(\alpha)$ we use Proposition~\ref{prop:non-empty} to construct a distribution in $\mathcal{L}$.

Fix some $i$. Recall that the function $t^i(\alpha)$ is defined as $\max\{t_\mathcal{L}^i(\beta)\mid F(\alpha,\beta)\}$. Let $c_1>c_2>\ldots> c_n$ be all the values of $t_{\mathcal{L}}$ and $U_k=\{\beta\mid t^i_{\mathcal{L}}(\beta)\ge c_k\}$. Let $V_k=F^{-1}(U_k)=\{\alpha\mid t^i(\alpha)\ge c_k\}$. The sets $U_k$ and $V_k$ increase as $k$ increases; all these sets are clopen.  The function $t^i (\alpha)$ equals $c_1,\ldots,c_n$ when $\alpha$ belongs to the sets $V_1$, $V_2\setminus V_1$,\ldots,$V_{n}\setminus V_{n-1}$ (respectively). Consider now the sets $V_1\times U_1$, $(V_2\setminus V_1)\times U_2$,\ldots, $(V_n\setminus V_{n-1})\times U_n$. On these sets $t^i(\alpha)\le t_{\mathcal{L}}^i(\beta)$. Consider the intersections of these sets with $F$; by construction their first projections coincide with $V_1$, $V_2\setminus V_1$,\ldots. Take the union of these intersections and apply Proposition~\ref{prop:non-empty} to it; we get a distribution whose first projection is $P$; its support is a subset of $F$ where $t^i(\alpha)\le t_{\mathcal{L}}^i(\beta)$. The second projection of this distribution is in $\mathcal{L}$ and it can be used to prove the inequality as described above. 

Similar argument can be used for probability-bounded tests: to show that $t'$ is probability-bounded, it is enough to show that all $t^{i}$ are; the set $V=\{\alpha\mid t^i(\alpha)>c\}$ is the $F$-preimage of the set $U=\{\beta\mid t_{\mathcal{L}}^i(\beta)>c\}$; we apply Proposition~\ref{prop:non-empty} to the set $(V\times U)\cup ((\Omega\setminus V)\times \Omega)$ and note that the inequality $\chi_V(\alpha)\le \chi_U(\beta)$ is true for every pair $\langle \alpha,\beta\rangle$ in this set.
\end{proof}

%\rd{Jason Rute: counterexample that show the necessity of these conditions on $F$??? Ask Jason!}

%\rd{\subsection*{Randomness for semimeasures?}

%Explanation of the problem and the example of Bienvenu, Porter et al.???}

% result about the closed set, randomness for semimeasures

%Missing things: 

\section*{Acknowledgments}

We would like to thank Bruno Bauwens, Wolfgang Merkle, Jason Rute and other participants of the Heidelberg Focus Semester,  Mikhail Andreev, Alexey Milovanov, Gleb Novikov, Andrey Rumyantsev, Nikolay Vereshchagin and other researchers from Kolmogorov seminar in Moscow, Bruno Durand, Andrei Romashchenko and other colleagues in the ESCAPE group in Montpellier, as well as ESCAPE visitors, especially Peter G\'acs, and all the people in the HSE and Poncelet laboratory in Moscow for their hospitality.}

\end{document}